\documentclass[11pt,a4paper]{amsart}

\pdfpkmode={lexmarkr}
\pdfmapfile{}

\usepackage{amsmath,amsfonts,mathrsfs,graphicx}

\DeclareMathOperator{\supp}{supp}

\newtheorem{theorem}{Theorem}
\newtheorem{corollary}{Corollary}
\newtheorem{proposition}{Proposition}
\newtheorem{lemma}{Lemma}
\theoremstyle{remark}
\newtheorem{assumption}{Assumption}

\begin{document}

\title{Typical points and families of expanding interval mappings}
\author{Tomas Persson} \address{Centre for Mathematical Sciences, Lund
  University, Box 118, 221 00 Lund, Sweden}
\email{tomasp@maths.lth.se} 
\urladdr{http://www.maths.lth.se/~tomasp/}
\date{\today}
\thanks{The author thanks D.~Schnellmann for useful comments.}

\begin{abstract}
  We study parametrised families of piecewise expanding interval
  mappings $T_a \colon [0,1] \to [0,1]$ with absolutely continuous
  invariant measures $\mu_a$ and give sufficient conditions for a
  point $X(a)$ to be typical with respect to $(T_a, \mu_a)$ for almost
  all parameters $a$. This is similar to a result by D.~Schnellmann,
  but with different assumptions.
\end{abstract}

\maketitle

\section{Introduction}

Assume that $I$ is an interval of parameters, and that for any $a \in
I$ we have a mapping $T_a \colon [-1,1] \to [-1,1]$ that is piecewise
expanding in a smooth and uniform way, and that $T_a$ depends on $a$
in a smooth way. It is a well known and classical result that such
mappings have invariant measures that are absolutely continuous with
respect to Lebesgue measure.

In \cite{schnellmann}, Schnellmann studied a class of such mappings
together with a point $X (a)$ in the domain of $T_a$. Under some
conditions on the family of mappings and on the function $a \mapsto
X(a)$, he proved that for almost all $a \in I$, the point $X (a)$ is
typical with respect to an invariant measure of $T_a$ that is
absolutely continuous with respect to Lebesgue measure. We say that a
point $x$ is typical with respect to $(T, \mu)$ if the sequence of
measures $\frac{1}{n} \sum_{k=0}^{n-1} \delta_{T_a^k (x)}$ converges
weakly to $\mu$ as $n \to \infty$.

In this paper we consider similar kind of mappings $T_a$ as
Schnellmann and prove a corresponding result, that $X (a)$ is typical
for almost all parameters $a$, but with assumptions that are different
than those used by Schnellmann. The proof of this result is based on
the proof by Schnellmann, but contains a new ingredient, that makes it
possible to remove one of the more restrictive assumptions used by
Schnellmann. Unfortunately, the proof also needs some new assumptions
on the mapping, so the result of this paper is not a generalisation of
Schnellmann's result, but it extends the result to some families of
mappings that were not covered by Schnellmann.

Almost sure typicality for families of piecewise expanding interval
mappings has also been proved by Schnellmann in
\cite[Theorem~3.5]{schnellmann2}, but with different assumptions and
with a method different from that used in \cite{schnellmann}. Because
of assumption (III) used by Theorem~3.5 in \cite{schnellmann2}, it is
somewhat unclear in what generality the result holds. (See also
Remark~4.2 of \cite{schnellmann2}.) Schnellmann proves that the
assumption (III) is satisfied for families of tent mappings,
generalised $\beta$-transformations and Markov mappings
\cite[Section~3]{schnellmann2}, but for other classes, the validity of
assumption (III) is unclear. In this paper we will prove results for
families that are not contained in Schnellmann's papers
\cite{schnellmann, schnellmann2}.

\section{Statement of the Result}

We start by stating more precisely what kind of mappings we will work
with. As mentioned before, $I$ is an interval of parameters, and for
any $a \in I$ we have a mapping $T_a \colon [-1,1] \to [-1,1]$ that
satisfies the following.

\begin{itemize}
  \item There are smooth functions $b_0, \ldots, b_p$ with
    \[
    -1 = b_0 (a) < b_1 (a) < \cdots < b_p (a) = 1
    \]
    for each $a \in I$, such that the restriction of $T_a$ to $(b_i
    (a), b_{i+1} (a))$ can be extented to a smooth and monotone
    function on some open neighbourhood of $[b_i (a), b_{i+1}
      (a)]$.
  \item There are numbers $1 < \lambda \leq \Lambda < \infty$ such
    that
    \[
    \lambda \leq |T_a' (x)| \leq \Lambda
    \]
    holds for all $a \in I$ and all $x \in [0,1] \setminus \{b_0 (a),
    \ldots, b_p (a) \}$. There is a number $L$ such that $T_a'$ is
    Lipschitz continuous with constant $L$ on each $[b_i (a), b_{i+1}
      (a)]$.
  \item For $x \in [0,1]$, the mappings $a \mapsto T_a (x)$ and $a
    \mapsto T_a' (x)$ are piecewise $C^1$.
\end{itemize}

For each $a$ there is a $T_a$-invariant probability measure that is
absolutely continuous with respect to Lebesgue measure, and there are
at most finitely many such measures.  In this paper we will work with
mappings for which there is a unique absolutely continuous invariant
probability measure $\mu_a$. Let $K (a) = \supp \mu_a$. By Wong
\cite{wong} or Kowalski \cite{kowalski}, $K (a)$ consists of finitely
many intervals. Schnellmann assumed that the endpoints of these
intervals depend in a smooth way on $a$. We will assume that $K (a) =
[-1,1]$ for all $a$. (More precisely, we assume that
Assumption~\ref{ass:weaklycovering} below holds, which implies $K(a) =
[-1,1]$.) This is not too restrictive, since this can be achieved by
restricting the mapping to $K (a)$ and a change of
variable. Smoothness of the family will then be preserved if $K(a)$
changes in a smooth way with the parameter $a$.

For a piecewise continuous mapping $T \colon [-1,1] \to [-1, 1]$ we
let $\mathscr{P} (T)$ denote the partition of $[-1,1]$ into the
maximal open intervals on which $T$ is continuous. We also write
$\mathscr{P}_j (a) = \mathscr{P} (T_a^j)$ and denote by $-1 = b_{0,j}
(a) < b_{1,j} (a) < \cdots < b_{p_j,j} (a) = 1$ the points such that
$\mathscr{P}_j (a) = \{(b_{i,n} (a), b_{i+1,n} (a))\}$.

We shall study the orbit of a point $X(a)$, and we write
\[
\xi_j (a) = T_a^j (X(a)).
\]
We assume throughout that $X$ is a $C^1$ function. Hence, the function
$\xi_j$ is piecewise smooth. We denote by $\mathscr{Q}_j$ the set of
maximal open intervals on which $\xi_j$ is smooth, that is the maximal
open intervals such that $\xi_i (a) \not \in \{b_0 (a), \ldots, b_{p}
(a)\}$ holds for $0 \leq i < j$.

We state below four assumptions on the family of mappings and the
function $X$. These conditions are the same as those used by
Schnellmann in \cite{schnellmann}.

\begin{assumption} \label{ass:der}
  There is a constant $C_0$ such that for any $j \geq 1$ and any
  $\omega \in \mathscr{Q}_j$ we have
  \[
  C_0^{-1} \leq \biggl| \frac{\xi_j'(a)}{(T_a^j)' (X(a))} \biggr| \leq
  C_0
  \]
  for all $a \in \omega$.
\end{assumption}

It may be difficult to check if Assumption~\ref{ass:der}
hold. Schnellmann proved that the following assumption implies
Assumption~\ref{ass:der}. (See Lemma~2.1 in \cite{schnellmann}.)

\begin{assumption} \label{ass:der2}
  $X$ is $C^1$, and there is a $j_0$ such that
  \[
  \inf_{a} |\xi_{j_0}' (a)| \geq \frac{\sup_{a, x} |\partial_a T_a
    (x)|}{\lambda - 1} + 2L.
  \]
\end{assumption}

Let $\phi_a$ denote the density of the absolutely continuous measure
$\mu_a$. Our next assumption concerns this density.

\begin{assumption} \label{ass:density}
  There is a constant $C_1$ such that for all $a \in I$ we have
  \[
  C_1^{-1} \leq \phi_a \leq C_1, \qquad \mu_a\text{-a.e.}
  \]
\end{assumption}

As is mentioned by Schnellmann in \cite{schnellmann}, the next
assumption is more restrictive than the above assumptions. It is the
purpose of the paper to extend Schnellmann's result to families of
mappings that do not necessarily satisfy this assumption.

\begin{assumption} \label{ass:symb}
  There is a constant $C_2$ such that for all $a_1, a_2 \in I$, $a_1
  \leq a_2$, and $j \geq 1$, there is a mapping
  \[
  \mathscr{U}_{a_1,a_2,j} \colon \mathscr{P}_j (a_1) \to \mathscr{P}_j
  (a_2),
  \]
  such that $\omega \in \mathscr{P}_j (a_1)$ and
  $\mathscr{U}_{a_1,a_2,j} (\omega)$ have the same symbolic dynamics,
  their images lie close in the sence that
  \[
  d (T_{a_1}^j (\omega), T_{a_2}^j (\mathscr{U}_{a_1,a_2,j} (\omega)))
  \leq C_2 |a_1 - a_2|,
  \]
  and
  \[
  |T_{a_1}^j (\omega)| \leq C_2 |T_{a_2}^j (\mathscr{U}_{a_1,a_2,j}
  (\omega))|.
  \]
\end{assumption}

Schnellmann proved that if the Assumptions \ref{ass:der},
\ref{ass:density} and \ref{ass:symb} are satisfied, then the point
$X(a)$ is typical for $(T_a,\mu_a)$ for Lebesgue almost every $a \in
I$. Assumption~\ref{ass:symb} is rather restrictive, and it would be
desireable to remove this assumption. We shall do so, but in doing so,
we will need to introduce the following two conditions instead.

If for any $\omega \in \mathscr{P} (a)$ there is an $N = N(a)$ such that
\[
  [-1,1] \setminus \bigcup_{n=0}^N T_a^n (\omega)
\]
is a finite set, then the mapping $T$ is called weakly covering by
Liverani in \cite{liverani}. According to Lemma~4.2 of that paper,
weakly covering implies that the mapping has a unique absolutely
continuous invariant measure, with density that is bounded and bounded
away from zero, and it is possible to give an explicit lower bound on
the density. We shall need such lower bounds that are also stable
under certains perturbations of the mapping. In order to achieve this
we will need to do as follows.

For $\omega \in \mathscr{P} (a)$, let $\tilde{T}_a (\omega) = T_a
(\omega)$. Suppose that $\tilde{T}_a^k (\omega)$ is defined. Then we
define
\[
\tilde{T}_a^{k+1} (\omega) = \bigcup_{\substack{\omega' \in \mathcal{P}
  (a),\\ \omega' \subset \tilde{T}_a^k (\omega)}} T_a (\omega').
\]
Note that $\tilde{T}_a^k (\omega) \subset T_a^k (\omega)$.

Our next assumption will be the following assumption that is stronger
than weakly covering.

\begin{assumption} \label{ass:weaklycovering}
  For any $\omega \in \mathscr{P} (a)$ there is an $N = N(a)$ such
  that
  \[
    [-1,1] \setminus \bigcup_{n=0}^N \tilde{T}_a^n (\omega)
  \]
  is a finite set.
\end{assumption}


\begin{assumption} \label{ass:largeimage}
  There is a number $\delta > 0$ and an integer $m \geq 1$ such that
  \[
  (-\delta, \delta) \subset T_a^m (b_{i,m} (a), b_{i+1,m} (a))
  \]
  for all $i$ and $a \in I$, and
  \[
  \delta > \frac{1}{\inf |(T_a^m)'| - 1}
  \]
  for all $i \in I$.
\end{assumption}

We shall prove the following.

\begin{theorem} \label{the:typicalpoints}
  Suppose that the family $T_a$ and the point $X \in C^1$ satisfies
  the Assumptions~\ref{ass:der2}, \ref{ass:weaklycovering} and
  \ref{ass:largeimage}. Then the point $X(a)$ is typical with respect
  to $(T_a, \mu_a)$ for Lebesgue almost every $a \in I$.
\end{theorem}

\section{Example}

Let $0 = b_0 < b_1 < b_2 < \cdots$ be an increasing and unbounded
sequence of real numbers. Suppose $T \colon [0,\infty) \to [0,1]$ is
  smooth on each of the intervals $(b_i,b_{i+1})$ and $|T'(x)| \geq
  \lambda_0 \geq 1$. We then define a family of mappings $T_a \colon
         [0,1] \to [0,1]$ by $T_a (x) = T (ax)$.

\begin{corollary} \label{cor:beta}
  Suppose that there exists a $\delta$ such that
  \[
  \Bigl(\frac{1 - \delta}{2}, \frac{1 + \delta}{2} \Bigr) \subset T
  ((b_i, b_{i+1}))
  \]
  for all $i$. Let
  \[
  a_0 = \lambda_0^{-1} (1 + \delta^{-1}).
  \]
  If $I \subset [a_0, \infty)$ is an interval such that for some $i$ we
    have
  \[
  (b_i / a, b_{i+1} / a) \subset \Bigl( \frac{1 - \delta}{2}, \frac{1
    + \delta}{2} \Bigr), \qquad \forall a \in I,
  \]
  and $T(b_i, b_{i+1}) = [0,1]$, then $X (a)$ is typical for a.e.\ $a
  \in I$ provided either 
  \begin{itemize}
    \item[i)] $X'(a) \geq 0$ and $T$ is piecewise increasing,
  \end{itemize}
  or
  \begin{itemize}
    \item[ii)] $X$ satisfies Assumption~\ref{ass:der2}.
  \end{itemize}
\end{corollary}

\begin{figure}
  \includegraphics{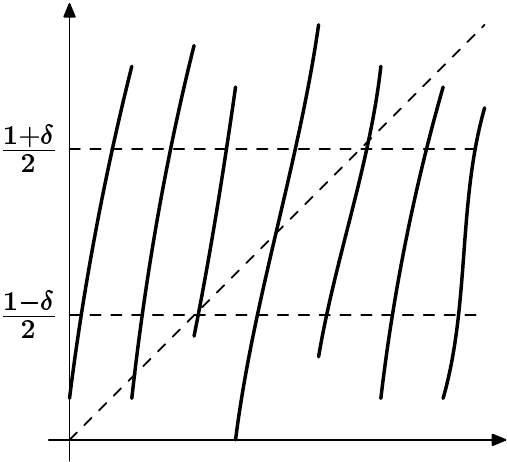}

  \caption{An example of a mapping $T$ for which the assumptions in
    Corollary~\ref{cor:beta} are satisfied for $T_a (x) = T(ax)$, for
    all parameters in some interval $[1,a_1]$, $a_1 > 1$. Here we have
    taken $\delta = 2/5$. Assumption~\ref{ass:largeimage} is then that
    $\inf |T_a'| > 7/2$.} \label{fig:example}
\end{figure}

Schnellmann considered also families of the form $T_a (x) = T(ax)$,
see Section~1.1 in \cite{schnellmann}, but used the assumption that $T
((b_i, b_{i+1}))$ is always of the form $[0,c)$, and that $T$ is
  piecewise expanding. Hence we can relax this assumption, but we have
  to add the other assumptions instead. Figure~\ref{fig:example} shows
  an example of a mapping satisfying the assumptions of the
  corollary.

\begin{proof}[Proof of Corollary \ref{cor:beta}]
  We shall use Theorem~\ref{the:typicalpoints}. Note that
  Theorem~\ref{the:typicalpoints} is for mappings on $[-1,1]$ and that
  here we are working on the interval $[0,1]$. This is just a matter
  of a change of variables.

  In case $X'(a) \geq 0$ and $T$ is piecewise increasing,
  Assumption~\ref{ass:der2} follows exactly as in Schnellmann's paper.

  The assumptions on $(b_i, b_{i+1})$ and $(b_i/a, b_{i+1}/a)$ in the
  corollary implies that Assumption~\ref{ass:weaklycovering}
  holds. Assumption~\ref{ass:largeimage} is satisfied, since the
  definition of $a_0$ implies that $\delta > 1/(\inf |T_a'| - 1)$ for
  $a \in I$. Hence, the conclusion follows from
  Theorem~\ref{the:typicalpoints}.
\end{proof}

\section{Outline of the proof}

The proof goes along the same line as the proof of Schellmann. Let $B$
be an interval and define
\[
F_n (a) = \frac{1}{n} \sum_{j=1}^n \chi_B (\xi_j (a)).
\]
It is sufficient to show that there exists a constant $C$ such that
\[
\limsup_{n\to\infty} F_n (a) \leq C |B|
\]
holds for any interval $B$ with rational endpoints. Using a lemma by
Bj\"ork\-lund and Schnellman \cite{bjorklundschnellmann} one shows that
it is sufficient to show that
\[
\int_{\tilde{I}} \chi_B (\xi_{j_1} (a)) \cdots \chi_B (\xi_{j_h} (a))
\, \mathrm{d}a \leq (C|B|)^h,
\]
for sufficiently sparse sequences $j_1, \ldots, j_h$, were $h$ is an
integer and $\tilde{I} \subset I$ is a small interval of parameters.
These estimates are stated in Proposition~\ref{pro:mainprop} below,
and are achieved by switching from an integral over the parameter
space to an integral over the phase space for a fixed mapping.

This is where the main difference with Schnellmann's paper
appears. Schnellmann switches the integral over the parameter space to
an integral over the phase space for a mapping $T_{a_J}$ in the
family. In order to be able to do so it is important that the orbit
structure of $T_{a_J}$ is rich enough to contain the orbits of $T_a$
for parameters $a$ that are close to $a_J$. This is where
Assumption~\ref{ass:symb} is necessary.

Here we will instead perturb the mapping $T_{a_J}$ to a mapping that is
close to $T_{a_J}$ but does not belong to the family. In this way we
can artificially make sure that a variant of Assumption~\ref{ass:symb}
is satisfied, even if the assumption itself is not satisfied for the
family. To prove that this is possible requires some new
assumptions. In effect, we can remove Assumption~\ref{ass:symb}, but
we have to replace it with the assumptions used in
Theorem~\ref{the:typicalpoints}.

In Section~\ref{sec:shifts} below, we will state and prove the results
that are necessary to get the desired properties of the above
mentioned perturbation of the mapping $T_{a_J}$. In
Section~\ref{sec:theproof}, we will prove
Theorem~\ref{the:typicalpoints}.

\section{Some preparations: On nested subshifts} \label{sec:shifts}

Let $D = [-1,1]$. We will consider piecewise expanding mappings on
$D$.

Given a vector of numbers $b = (b_0, b_1, \ldots, b_n)$ with $-1 = b_0
< b_1 < \cdots < b_n = 1$ and a vector of functions $f = (f_0, f_1,
\ldots, f_n)$ with $f_k \colon [b_{k-1}, b_k] \to D$, we define a
mapping $S \colon (b,f) \mapsto T$, where $T$ is a mapping $T \colon D
\to D$ such that $T (x) = f_k (x)$ for $x \in (b_{k-1}, b_k)$. We
leave $T$ undefined at the points $b_k$, and let $D_k = (b_{k-1},
b_k)$.

With the mapping $T$, we associate the shift space
\[
\Sigma (T) = \{\, i \in \{1, 2, \ldots, n\}^\mathbb{N} : \exists x \in
D, \ T^k (x) \in D_{i_k} \,\}.
\]

Suppose we have smooth functions $b_0, b_1, \ldots, b_n$ defined on
$[0, \varepsilon)$, such that
\[
-1 = b_0 (t) < b_1 (t) < \cdots < b_n (t) = 1
\]
and
\[
| b_k' (t)| \leq \zeta
\]
for all $t \in [0, \varepsilon)$, where $\zeta$ is a fixed number. We
  will use the notation $b_t = (b_0(t), b_1 (t), \ldots, b_n (t))$ and
  $D_k (t) = (b_{k-1} (t), b_k (t))$.

Let $f_1, f_2, \ldots, f_n$ be smooth mappings with $f_k \colon [0, P)
  \times D \to D$ such that for any $t \in [0, P)$ we have
\begin{align*}
  \lambda \leq | \partial_x f_k (t, x) | &\leq \Lambda, && \text{for
    all } x \in D_k (t), \\ | \partial_t f_k (t, x) |
  &\leq \eta, && \text{for all } x \in D_k (t),
\end{align*}
where $1 < \lambda \leq \Lambda < \infty$ and $\eta$ are fixed
numbers. We also assume that if $f_k (t_0, b_k (t_0)) \in \{-1, 1\}$
for some $t_0$, then $f_k (t, b_k (t)) = f_k (t_0, b_k(t_0))$ for all
$t$, and similarly if $f_k (t_0, b_{k-1} (t_0)) \in \{-1, 1\}$.

Actually, we only assume that the functions $f_k$ are defined
for $(t,x)$ such that $x \in D_k (t)$. We will consider $t$ as our
parameter and write $f_{k,t}$ for the function $x \mapsto f_k (t, x)$,
and we define the vector $f_t = (f_{1,t}, f_{2,t}, \ldots, f_{n,t})$.

We now define mappings $E_s$. Given a mapping $f_k (t, \cdot) \colon
D_k (t) \to [-1,1]$ and a number $s$ close to one, we define
\[
E_s (f_k (t,\cdot)) = \tilde{f}_k \colon D_k (t) \to [-1,1],
\]
where $\tilde{f}_k$ is defined by
\begin{equation} \label{eq:Edefinition}
  \tilde{f}_k (x) = \left\{ \begin{array}{ll} s f_k (t, x) & \text{if
    } \overline{f_k (t, D_k(t))} \subset (-1,1), \\ f_k (t, x) &
    \text{if } f_k (t, D_k(t)) = (-1,1), \\ \frac{s+1}{2} f_k (t, x) +
    \frac{s-1}{2} & \text{if } -1 \in \overline{f_k (t, D_k(t))}
    \subset [-1,1) ,\\ \frac{s+1}{2} f_k (t, x) - \frac{s-1}{2} &
      \text{if } 1 \in \overline{f_k (t, D_k(t))} \subset (-1,1].
  \end{array} \right.
\end{equation}
We shall only consider $E_s$ for $s \geq 1$. Hence $E_s$ maps
$f_k(t,\cdot)$ into $\tilde{f}_k$, so that the graph of $\tilde{f}_k$
is the graph of $f_k (t, \cdot)$, expanded in such a way that the
image stays in $(-1,1)$, see Figure~\ref{fig:expand}.

\begin{figure}
  
  \includegraphics{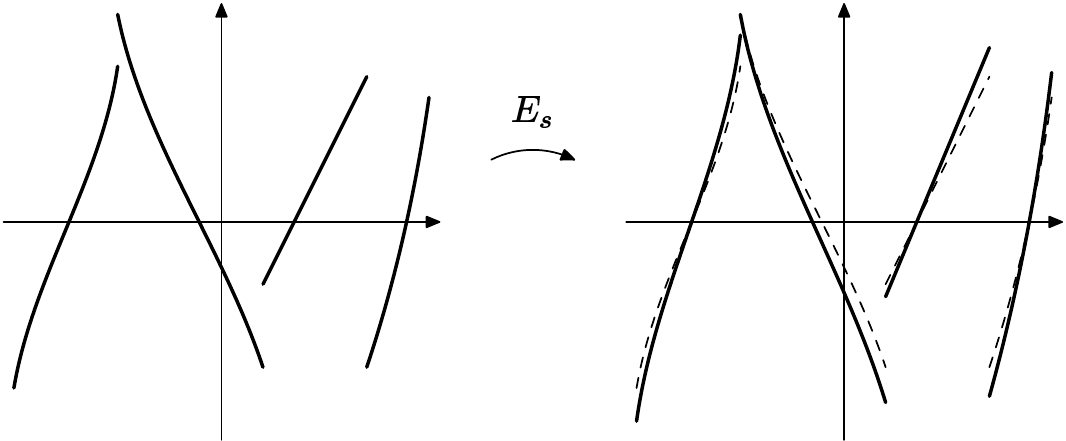}

  \caption{An illustration of the action of $E_s$ with $s =
    1.2$. The dashed lines show the original graph.} \label{fig:expand}
\end{figure}

When $f_t$ is a vector $f_t = (f_{1,t}, f_{2,t}, \ldots, f_{n,t})$ we
define $E_s (f_t)$ by
\[
E_s (f_t) = (E_s (f_{1,t}), E_s (f_{2,t}), \ldots, E_s (f_{n,t})).
\]

We now define the mappings $T_{t, s} \colon I \to I$ for any $t \in
[0, P)$ and $s$ close to one. Recall that $S \colon (b,f) \mapsto T$
  was defined in the beginning of this section and define
\[
T_{t,s} = S (b_t, E_s (f_t)).
\]
We also define the symbolic spaces $\Sigma_{t,s}$ by
\[
\Sigma_{t,s} = \Sigma (T_{t,s}).
\]

We are going to prove the following.

\begin{proposition} \label{pro:nestedspaces}
  Assume that $(-\delta, \delta) \subset f_k (t,
  D_k)$ for all $k$ and $t$, where $\delta$ is a number such that
  \[
  \delta > \frac{1}{\lambda - 1}.
  \]
  Then there are numbers $\alpha_0$ and $d$, depending only on
  $\lambda$, $\Lambda$ and $\zeta$, such that
  \[
  \left. \begin{array}{r} \alpha > \alpha_0 \\ 0 \leq t_0 < t_1 <
    d/\alpha \end{array} \right\} \quad \Rightarrow \quad \Sigma_{t_0,
    1 + \alpha t_0} \subset \Sigma_{t_1, 1 + \alpha t_1}.
  \]
  Moreover, we have
  \[
  T_{t_0, 1 + \alpha t_0}^j (\omega) \subset T_{t_1, 1 + \alpha t_1}^j
  (\mathscr{U}_{t_0,t_1,j} (\omega))
  \]
  for any $\omega \in \mathscr{P}_j (T_{t_0,1 + \alpha t_0})$, where
  $\mathscr{U}_{t_0,t_1,j}$ is as in Assumption~\ref{ass:symb}.
\end{proposition}

The assumption that $(- \delta, \delta) \subset f_k (t, D_k (t))$ for
any $k$ is, as it is stated here, only here for convenience in the
proof. One can make small adjustments in the definition of $T_{t,s}$
so that the conclusion of Proposition~\ref{pro:nestedspaces} holds
with various variations of this assumption. It is not important that
$0 \in f_k (t, D_k (t))$, but it seems that some kind of large image
property is needed.

\begin{proof}
  Fix $t_0$ and $t_1$ in $[0, \varepsilon) \subset [0, P)$ such that
      $t_0 < t_1$. Write $s (t) = 1 + \alpha t$.

  It sufficies to prove the following. For any $m \in \mathbb{N}$, $i
  \in \Sigma_{t_0, s (t_0)}$, and $x_0, x_1, \ldots, x_m$ such that
  \begin{equation} \label{eq:assumptiononx0}
    x_k = T_{t_0, s (t_0)} (x_{k-1}) = T_{t_0, s (t_0)}^k (x_0) \in
    D_{i_k} (t_0), \quad k = 1, 2, \ldots, m,
  \end{equation}
  there exist $y_0, y_1, \ldots, y_m$ such that
  \[
  y_k = T_{t_1, s (t_1)} (y_{k-1}) = T_{t_1, s (t_1)}^k (y_0) \in D_{i_k}
  (t_1), \quad k = 1, 2, \ldots, m.
  \]

  Let $J_j (t) = T_{t,s(t)} (D_j (t)) = E_{s (t)} (f_{j,t} (D_j
  (t)))$. We could try to simply define $y_k$ by
  \begin{align*}
    y_m &= x_m, \\ y_{k-1} &= (E_{s (t_1)} (f_{i_{k-1}, t_1}))^{-1}
    (y_k),
  \end{align*}
  but a possible obstruction is that $y_k$ might not lie in
  $J_{i_{k-1}} (t_1)$, and in that case $(E_{s (t_1)} (f_{i_{k-1},
    t_1}))^{-1} (x_k)$ is not defined. We will show that our
  assumptions imply that this obstruction is not present.

  We let
  \begin{align*}
    y_m (t) &= x_m, \\ y_{k-1} (t) &= (E_{s (t)} (f_{i_{k-1},
      t}))^{-1} (y_k (t)),
  \end{align*}
  and proceed by induction. Cleary, $y_0 (t), y_1 (t), \ldots, y_m (t)$
  are defined for $t = t_0$, and since the intervals $J_j (t)$ are
  open, it follows that $y_0 (t), y_1 (t), \ldots, y_m (t)$ are
  defined in some neighbourhood around $t = t_0$. We want to estimate
  the size of that neighbourhood, and for future reference, we denote
  it by $L$.

  Clearly, $y_m (t) = x_m$ is defined for any $t$, and $|\partial_t
  y_m (t)| = 0 < K$, where $K$ is a positive constant to be determined
  later. Assume that for some $k$ the point $y_{k}$ is defined and
  that $|\partial_t y_k (t)| \leq K$. Then, since $y_{k-1} (t) =
  (E_{s(t)} (f_{i_{k-1}, t}))^{-1} (y_k (t))$, we have
  \[
  E_{s(t)} (f_{i_{k-1}}) (t, y_{k-1} (t)) = y_k (t).
  \]

  We will now consider the different cases in the definition of $E_s$
  in \eqref{eq:Edefinition}. Consider the first case, in which we have
  \[
  E_{s(t)} (f_{i_{k-1}}) (t, y_{k-1} (t)) = s(t) f_{i_{k-1}} (t,
  y_{k-1} (t)) = y_k (t).
  \]
  After a differentiation, we get
  \begin{multline*}
    y_k' (t) = s' (t) f_{i_{k-1}, t} (y_{k-1} (t)) + s(t) \partial_t
    f_{i_{k-1}} (t, y_{k-1} (t)) \\ + s(t) \partial_x f_{i_{k-1}} (t,
    y_{k-1} (t)) y_{k-1}' (t).
  \end{multline*}
  Solving for $y_{k-1}' (t)$ and using that $s'(t) = \alpha$, we get
  \[
    y_{k-1}' (t) = \frac{( y_k' (t) - \alpha f_{i_{k-1}, t}
      (y_{k-1} (t)) - s(t) \partial_t f_{i_{k-1}} (t, y_{k-1}
      (t))}{s(t) \partial_x f_{i_{k-1}} (t, y_{k-1} (t))}.
  \]
  Hence
  \begin{equation} \label{eq:diffyestimate}
  |y_{k-1}' (t)| \leq \lambda^{-1} (K + \alpha + (1 + \alpha
  \varepsilon) \eta).
  \end{equation}

  Similar calculations for the other cases in \eqref{eq:Edefinition}
  yields the same estimate. Hence, in all cases we will have
  \eqref{eq:diffyestimate}. Since $\lambda > 1$, it is clear that
  this implies that
  \[
  |y_{k-1}' (t)| \leq K
  \]
  if $K$ satisfies $K \geq (\lambda - 1)^{-1} ( \alpha + (1 + \alpha
  \varepsilon) \eta)$. Therefore, put $K = (\lambda - 1)^{-1} (
  \alpha + (1 + \alpha \varepsilon) \eta)$.

  We have now proved that for any $t \in L$, the points $y_k (t)$ are
  defined and $|y_k' (t)| \leq K$, but we still do not know how big
  $L$ is.

  What can prevent the neighbourhood $L$ to be large is that for some
  $k$ and $t$, the point $y_k (t)$ is not in the interval $J_{i_{k-1}}
  (t)$. However, as $t$ varies, $y_k (t)$ moves at a speed not larger
  than $K$, and at the same time, as $t$ grows, the interval
  $J_{i_{k-1}} (t)$ expands and we would like to know that the
  endpoints of $J_{i_{k-1}} (t)$ moves with a speed greater than $K$,
  so that $y_k (t)$ cannot escape out from the interval $J_{i_{k-1}}
  (t)$ as $t$ grows. To show that this is the case we need to consider
  the four cases in \eqref{eq:Edefinition}.

  Let $c (t) = E_{s(t)} (f_{i_{k-1}}) (t, b_{i_{k-1}} (t))$ be one of
  the endpoints of $J_{i_{k-1}} (t)$. (The other endpoint can be
  treated in the same way, but we will not do so.) In the case that
  $f_{i_{k-1}} (t, b_{i_{k-1}} (t)) \in \{-1,1\}$, the interval
  $J_{i_{k-1}} (t)$ will be maximal at one of it's end points, and the
  point $y_k (t)$ can therefore not escape from $J_{i_{k-1}} (t)$ at
  that endpoint. It is therefore sufficient to only consider the other
  cases.

  In the case that
  \[
  c (t) = E_{s(t)} (f_{i_{k-1}}) (t, b_{i_{k-1}} (t)) = s(t)
  f_{i_{k-1}} (t, b_{i_{k-1}} (t))
  \]
  we have that
  \begin{multline*}
    c' (t) = s' (t) f_{i_{k-1}} (t, b_{i_{k-1}} (t)) + s (t)
    \partial_t f_{i_{k-1}} (t, b_{i_{k-1}} (t)) \\ + s(t) \partial_x
    f_{i_{k-1}} (t, b_{i_{k-1}} (t)) b_{i_{k-1}}' (t).
  \end{multline*}
  We estimate each of the three terms separately:
  \begin{align*}
    |s' (t) f_{i_{k-1}} (t, b_{i_{k-1}} (t))| & \geq \alpha \delta,
    \\ | s(t) \partial_t f_{i_{k-1}} (t, b_{i_{k-1}} (t))| & \leq (1 +
    \alpha \varepsilon) \eta, \\ | s(t) \partial_x f_{i_{k-1}} (t,
    b_{i_{k-1}} (t)) |b_{i_{k-1}}' (t)| & \leq (1 + \alpha
    \varepsilon) \Lambda \zeta.
  \end{align*}
  This yields
  \[
  |c'(t)| \geq \alpha \delta - (1 + \alpha \varepsilon) (\eta +
  \Lambda \zeta)
  \]
  and we get the same estimate in the cases
  \[
  c (t) = E_{s(t)} (f_{i_{k-1}}) (t, b_{i_{k-1}} (t)) = \frac{s(t)+1}{2}
  f_{i_{k-1}} (t, b_{i_{k-1}} (t)) \pm \frac{s(t)-1}{2}.
  \]

  Now, $|c'(t)| > K \geq |y_k'(t)|$ if
  \begin{align*}
  \alpha \delta - (1 + \alpha \varepsilon) (\eta + \Lambda \zeta) &>
  \frac{1}{\lambda - 1} (\alpha + (1 + \alpha \varepsilon) \eta)
  \\ &\Leftrightarrow \\ \alpha \Bigl( \delta - \frac{1}{\lambda - 1}
  \Bigr) & > (1 + \alpha \varepsilon) \Bigl( \frac{\lambda
    \eta}{\lambda - 1} + \Lambda \zeta \Bigr).
  \end{align*}
  From this, it appears that if
  \[
  \delta > \frac{1}{\lambda - 1},
  \]
  then there exists a number $\alpha_0$ such that $|c' (t)| > |y_k'
  (t)|$ for $\alpha \geq \alpha_0$ provided that $\alpha \varepsilon
  \leq 1$.

  Moreover, since we require that $T_t,s ([-1,1]) \subset [-1,1]$, the
  mappings $E_s$ will only be defined for $1 \leq s \leq s_0$, where
  $s_0 > 1$ only depends on $f_t$. Hence we also require that $1 +
  \alpha \varepsilon \leq s_0$. In conclusion, the conclusion of the
  proposition holds provided
  \[
  \varepsilon < \frac{\min \{ 1, s_0 - 1\}}{\alpha}. \qedhere
  \]
\end{proof}

\section{Proof of Theorem~\ref{the:typicalpoints}} \label{sec:theproof}

The proof is a modification of Schnellmann's proof of his result. We
shall therefore make use of several of the lemmata that are found in
Schnellmann's paper \cite{schnellmann}. For the proof of these lemmata
we refer to Schnellmann's paper. However, we shall repeat here the
steps in the proof of Schnellmann's result that coincide with our
proof of Theorem~\ref{the:typicalpoints}.

\subsection{Densities of Invariant Measures}

We first note that the Assumption~\ref{ass:weaklycovering} implies
that $T_a$ is weakly covering, and this implies that the density
$\phi_a$ of $\mu_a$ satisfies
\[
\phi_a (x) \geq \gamma_a > 0,
\]
for some $\gamma_a$. This is Lemma~4.3 in Liverani's paper
\cite{liverani}. In fact, the proof of that lemma gives
\begin{equation} \label{eq:densitybound}
  \gamma_a \geq 2^{-2} \lVert T_a \rVert_\infty^{-N}.
\end{equation}

It is also clear from Rychlik's paper \cite{rychlik} that $\phi_a$ is
of bounded variation for all $a$, and that there is a uniform bound on
the variation of $\phi_a$. In particular, the functions $\phi_a$ are
uniformly bounded. Hence there exists a constant $C_0$ such that
\[
\gamma_a \leq \phi_a (x) \leq C_0, \qquad \forall x,
\]
holds for all $a \in I$.

We also note that if $T_a$ satisfies
Assumption~\ref{ass:weaklycovering}, then so does $E_s T_a$ for any $s
\geq 1$, where $E_s T_a$ is defined in Section~\ref{sec:shifts}. This
follows immediately since $T (P) \subset (E_s T_a) (P)$ for any $P \in
\mathscr{P} (a) = \mathscr{P} (E_s T_a)$. Note however, that we do not
necessarily have $T^n (P) \subset (E_s T_a)^n (P)$ for $n > 1$. This
is the only reason that we use Assumption~\ref{ass:weaklycovering}
instead of just assuming that $T$ is weakly covering.

By Lemma~4.3 in \cite{liverani}, it then follows that $E_s
T_a$ has an invariant probability measure $\mu_{a,s}$ that is
equivalent to Lebesgue measure.

Let $\phi_{a,s}$ denote the density of $\mu_{a,s}$. Using the explicit
bound on $\gamma_a$ from below \eqref{eq:densitybound}, we can
conclude that
\begin{equation} \label{eq:densityestimate}
  \frac{1}{2} \gamma_a \leq \phi_{a,s} (x) \leq C_0, \qquad \forall x,
\end{equation}
holds for all $a \in I$ and $1 \leq s \leq s_0$.

It would be useful if \eqref{eq:densityestimate} would give a bound
from below that is uniform in $a$, since then
Assumption~\ref{ass:density} would be satisfied, but this need not be
the case. However, we can take a subset $I'$ of $I$ of almost full
measure and a $\gamma > 0$ such that
\begin{equation} \label{eq:densityestimate2}
  \gamma \leq \phi_{a,s} (x) \leq \gamma^{-1}, \qquad \forall x,
\end{equation}
holds for all $a \in I'$ and $1 \leq s \leq s_0$. We then have
Assumption~\ref{ass:density} for $I'$ replaced by $I$, and by working
with $I'$ instead of $I$, we can perform the proof in the same way as
if \eqref{eq:densityestimate2} had been satisfied for all $a \in I$.

The possibility to work with $I'$ instead of $I$ was remarked already
by Schnellmann, see  Remark~2 in \cite{schnellmann}.

Instead of working with $I'$ we shall however work with $I$ and assume
that \eqref{eq:densityestimate2} holds for all $a \in I$; This is
mostly a typographical difference, and one can throughout the proof
exchange $I$ with $I'$. Subintervals $\tilde{I}$ of $I$ should be
replaced with $\tilde{I} \cap I'$, et\,c.

\subsection{Switching from the parameter space to the phase space} \label{sec:parameter-phase}

Assumption~\ref{ass:der} was used by Schnellmann to switch from
integrals over the parameter space to integrals over the phase
space. We shall also need to do so, but we will need that
Assumption~\ref{ass:der} holds also for the mappings $E_s T_a$ with $s = 1
+ (a - a_0) \alpha$. We can achieve this as follows.

Assumption~\ref{ass:der2} implies Assumption~\ref{ass:der}, as
previously mentioned. Hence, $|\xi_j' (a)|$ grows axponentially fast
with $j$. If $T_a^j (X(a))$ is typical, then so is $X (a)$. Therefore,
working with $\tilde{X} (a) = T_a^j (X(a))$ instead of $X(a)$, we can
achive that $|\xi_1' (a)|$ is as large as we desire. Since
\[
\frac{\sup_{a, x} | \partial_a (E_s T_a (x))|}{\lambda - 1} + 2L
\]
is bounded, Assumption~\ref{ass:der2} will be satisfied for $E_s T_a$,
and hence also Assumption~\ref{ass:der}.

\subsection{Typical points}

We let $\mathscr{B}$ denote the set of open sub intervals of $[0,1]$
with rational endpoints. The strategy of the proof is to show that
there is a constant $C$ such that for any $B \in \mathscr{B}$, the
function
\[
F_n (a) = \frac{1}{n} \sum_{j=1}^n \chi_B (\xi_j (a))
\]
satisfies
\begin{equation} \label{eq:boundeddensity}
\limsup_{n \to \infty} F_n (a) \leq C |B|, \qquad \text{for a.e. } a
\in I.
\end{equation}
This is sufficient to prove Theorem~\ref{the:typicalpoints}, since
then there is a set of parameters of full measure for which $\limsup
F_n (a) \leq C |B|$ holds for every $B$, and this implies that any
weak accumulation point of
\[
\frac{1}{n} \sum_{j=1}^n \delta_{\xi_j (a)}
\]
is an absolutely continuous invariant measure with density bounded by
$C$.

At this point, we shall assume that the constant $m$ appearing in
Assumption~\ref{ass:largeimage} is equal to 1. If, instead of $T_a$,
we consider the family $T_a^m$, then, as we shall see below, we can
prove that $X(a)$ is typical with respect to $(T_a^m, \mu_a)$ for
almost all $a$ by proving \eqref{eq:boundeddensity} for $T_a^m$
instead of $T_a$. This result then holds for all of the points $X(a),
T_a (X(a)), \ldots, T_a^{m-1} (X(a))$, so that
\eqref{eq:boundeddensity} holds for $T_a$ as well. Hence, we may
assume that $m=1$.

To prove \eqref{eq:boundeddensity}, Schnellmann used a lemma by
Bj\"orklund and Schnellmann \cite{bjorklundschnellmann}: It is
sufficient to show that for all large integers $h$ there is a constant
$C_1$ and an integer $n_{h,B}$, growing at most exponentially fast
with $h$, such that
\[
\int_I F_n (a)^h \, \mathrm{d}a \leq C_1 (C|B|)^h,
\]
for all $n \geq n_{h,B}$.

We can write
\begin{equation} \label{eq:Fnintegral}
  \int_I F_n (a)^h \, \mathrm{d}a = \sum_{1 \leq j_1, \ldots, j_h \leq
    n} \frac{1}{n^h} \int_I \chi_B (\xi_{j_1} (a)) \cdots \chi_B
  (\xi_{j_h} (a)) \, \mathrm{d}a.
\end{equation}
The idea is then to compare the integral over the parameters with
integrals over the phase space $[0,1]$ with respect to $\mu_a$, and
use mixing to achieve the desired estimate. Indeed, for a fixed $a$,
there is a set $A$ and a number $k$ such that $T_a^k \colon A \to A$
is mixing of any order (see \cite{wagner}). We therefore have
\begin{equation} \label{eq:mixing}
  \int_A \chi_B (T_a^{kj_1} (x)) \cdots \chi_B (T_a^{k j_h}) \,
  \mathrm{d} \mu_a (x) \to \mu_a (B)^h \mu_a (A) \leq (\gamma^{-1} |B|)^h,
\end{equation}
as $j_i \to \infty$ and $\min_{i \neq l} |j_i - j_l| \to \infty$.

By comparing the integrals in \eqref{eq:Fnintegral} and
\eqref{eq:mixing}, we shall prove the following.

\begin{proposition} \label{pro:mainprop}
  The set $I$ of parameters can be covered by countably many intervals
  $\tilde{I} \subset I$, such that for each $\tilde{I}$ there is a
  constant $C$ and numbers $n_{h,B}$, growing at most exponentially
  fast with $h$, such that
  \[
  \int_{\tilde{I}} \chi_B (\xi_{j_1}(a)) \cdots \chi_B (\xi_{j_h} (a))
  \, \mathrm{d}a \leq (C |B|)^h,
  \]
  for all $(j_1, \ldots, j_h)$ with $\sqrt{n} \leq j_1 < j_2 < \cdots
  < j_h < n -\sqrt{n}$ and $j_i - j_{i-1} \geq \sqrt{n}$, were $n \geq
  n_{h,B}$.
\end{proposition}

Schnellmann proved the above proposition using also
Assumption~\ref{ass:symb}.

Let us now see how Proposition~\ref{pro:mainprop} finishes the proof
of Theorem~\ref{the:typicalpoints}. The number of $h$-tuples
$(j_1,\ldots,j_h)$ not satisfying the assumptions of
Proposition~\ref{pro:mainprop}, that is, the number of increasing
$h$-tuples satisfying $j_1 < \sqrt{n}$, $j_h > n - \sqrt{n}$ or
$j_{k+1} - j_k < \sqrt n$ for some $k$, is at most $2h n^{h - 1/2}$.
Hence, by \eqref{eq:Fnintegral} and Proposition~\ref{pro:mainprop} we
get
\[
\int_{\tilde{I}} F_n (a)^h \, \mathrm{d} a \leq (C|B|)^h + \frac{2
  h}{\sqrt{n}} | \tilde{I} | \leq 2 (C|B|)^h
\]
if $n \geq n_{h,B}$ and
\[
n \geq \frac{4 h^2 |\tilde{I}|^2}{(C|B|)^{2h}}.
\]
This implies that \eqref{eq:boundeddensity} holds for almost every $a$
in $\tilde{I}$. Hence, it remains only to prove
Proposition~\ref{pro:mainprop}.

\subsection{Proof of Proposition~\ref{pro:mainprop}}

We shall first state the following lemma. Suppose $J = [a_0, a_1]
\subset \tilde{I}$. Consider $t \mapsto T_{a_0+t}$ and take $\alpha$
according to Proposition~\ref{pro:nestedspaces}. Note that we can
choose $\alpha$ to be independent of $J$.

Put $s(a) = 1 + (a-a_0)\alpha$. In the following lemma, we will
consider the mappings $E_{s(a)} T_a$. For a partition $\mathscr{Q}$,
we denote by $\mathscr{Q} | J$ the partition $\mathscr{Q}$ restricted
to the interval $J$, that is the set of non-empty intersections
$\omega \cap J$, with $\omega \in \mathscr{Q}$.

\begin{lemma} \label{lem:q-to-one}
  Assume that Assumption~\ref{ass:der2} holds. There is an integer $q$
  and a constant $C_2$ such that if $J \subset \tilde{I}$ is of length
  about $1/n$, then there is a mapping
  \[
  \mathscr{U}_J \colon \mathscr{Q}_n | J \to \mathscr{P} ((E_{s (a_1)}
  T_{a_1})^n)
  \]
  that is at most $q$-to-one, and we have
  \[
  d (\xi_j (\omega), (E_{s (a_1)} T_{a_1})^j (\mathscr{U}_J (\omega))) \leq
  C_2 /n, \qquad 0 \leq j \leq n - \sqrt{n}
  \]
  and
  \[
  |\omega| \leq C_2 |\mathscr{U}_J (\omega)|.
  \]
\end{lemma}

This lemma is the main difference compared to Schnellmann's
paper. Schnellmann proved a corresponding lemma (Lemma~3.1 in his
paper) using also Assumption~\ref{ass:symb}.

\begin{proof}
  In fact, Lemma~\ref{lem:q-to-one} follows from Lemma~3.1 in
  \cite{schnellmann} combined with Proposition~\ref{pro:nestedspaces}
  as follows.  Let $S_a = E_{s(a)} T_a$ with $s(a) = 1 +
  (a-a_0)\alpha$. Then Assumption~\ref{ass:der2} holds for the family
  $S_a$, as noted in Section~\ref{sec:parameter-phase}.

  According to \eqref{eq:densityestimate2},
  Assumption~\ref{ass:density} is satisfied for the family $S_a$.

  By Proposition~\ref{pro:nestedspaces}, the family $S_a$ satisfies
  the assumptions of Lemma~3.1 of \cite{schnellmann}. The statement of
  that Lemma is exactly what is to be proved.
\end{proof}

From now on, we will work a lot with $E_{s(a_1)} T_{a_1}$. Therefore,
to simplify the notation, we put $T = E_{s (a_1)} T_{a_1}$.

Take $t_0$ such that $2^{1/t_0} \leq \sqrt{\lambda}$ and let
\[
\delta (a) = \min \{ \, |\omega| : \omega \in \mathscr{P}_{t_0} (a) \,
\}.
\]
It is clear that $\delta (a)$ depends continuously on $a$, so $\delta$
is locally bounded away from zero.

We show that there is a constant $C$ and $n_{h,B}$ such that if $J
\subset \tilde{I}$ is an interval of length between $1/(2n)$ and
$1/n$, then
\begin{equation} \label{eq:Jintegral}
  \int_J \chi_B (\xi_{j_1} (a)) \cdots \chi_B (\xi_{j_h} (a)) \,
  \mathrm{d}a \leq |J| (C|B|)^h,
\end{equation}
for $n \geq n_{h,B}$ and $(j_1, \ldots, j_h)$ satisfying the
assumptions of Proposition~\ref{pro:mainprop}. By covering $\tilde{I}$
with such intervals $J$, this implies the statement of
Proposition~\ref{pro:mainprop}.

Let
\[
\Omega_J = \{ \, \omega \in \mathscr{Q}_n | J : \xi_{j_i} (\omega)
\cap B \neq \emptyset, \ \forall 1 \leq i \leq h \, \},
\]
We show that
\begin{equation} \label{eq:Omegaestimate}
  |\cup \Omega_J| \leq |J| (C|B|)^h,
\end{equation}
which implies \eqref{eq:Jintegral}, since
\[
\{\, a \in J : \chi_B (\xi_{j_1} (a)) \cdots \chi_B (\xi_{j_h} (a)) =
1 \, \} \subset \cup \Omega_J.
\]

Let $J_X$ be the interval of length $|X(J)| + 3 C_2 /n$ and concentric
with $X(J)$. Put
\[
\Omega = \{\, \omega \in \mathscr{P}_n (a_1)|J_X : T^{j_i}
(\omega) \cap 2B \neq \emptyset,\ 1 \leq i \leq h \,\}.
\]

Now, using Lemma~\ref{lem:q-to-one}, we have
\[
\cup \mathscr{U}_J (\Omega_J) \subset \cup \Omega
\]
and
\[
|\cup \Omega_J| \leq C_2 |\cup \mathscr{U}_J (\Omega_J)| \leq C_2
|\cup \Omega|.
\]

Take $\tau$ such that $\Lambda^{-\tau} \leq |B| / 2$ and assume that
$n \geq \tau^2$. Then $j_i - j_{i-1} \geq \tau$. Let $\Omega_0 =
\{J_X\}$ and
\[
\Omega_i = \{\, \omega \in \mathscr{P}_{j_i + \tau} (a_1) | \cup
\Omega_{i-1} : T^{j_i} (\omega) \cap 2B \neq \emptyset \,\}, \qquad 1 \leq
i \leq h.
\]
We then have $\cup \Omega \subset \cup \Omega_h$.

By the way $\tau$ was chosen, we have $|T^{j_i} (\omega)| \leq |B|/2$
for all $\omega \in \mathscr{P}_{j_i + \tau} (a_1)$. Therefore
\[
\cup \Omega_i \subset \{ \, x \in \Omega_{i-1} : T^{j_i} (x) \in 3 B
\,\}.
\]

We let $\phi$ denote the density of the absolutely continuous
invariant measure of $T$. By the invariance of the density $\phi$, we
have
\[
\phi (x) = \sum_{T^k (y) = x} \frac{\phi (y)}{|(T^k)' (y)|}.
\]
Hence, using \eqref{eq:densityestimate2}, which states that $\gamma
\leq \phi \leq \gamma^{-1}$, we get
\[
\sum_{T^k (y) = x} \frac{1}{|(T^k)' (y)|} \leq
\gamma^{-2} \qquad \text{for a.e. } x.
\]
It therefore follows that
\begin{align*}
  |T^j (\{x \in \omega : T^{j+k} (x) \in 3B\} )| &= \int_{3B}
  \sum_{\substack{x \in \omega,\\ T^{j+k} (x) = y}} \biggl|
  \frac{(T^j)' (x)}{(T^{j+k})' (x)} \biggr| \, \mathrm{d}y \nonumber
  \\ &\leq \int_{3B} \sum_{T^k (x) = y} \frac{1}{|( T^k)' (x)|} \,
  \mathrm{d}y \leq 3 \gamma^{-2} |B|.
\end{align*}

There is a constant $C_3$ such that
\[
\biggl| \frac{(T^j)' (x_1)}{(T^j)' (x_2)}
\biggr| \leq C_3
\]
holds whenever $x_1,x_2 \in \omega \in \mathscr{P}_j (a_1)$. This is
just a standard distortion estimate. See Lemma~4.1 in Schnellmann's
paper for a more general result.

If $|T^j (\omega)| \geq \delta_0$, then we get
\begin{align} \label{eq:hit3B2}
  |\{\, x \in \omega : T^{j+k} (x) \in 3B \,\}| &\leq C_3
  \frac{|T^j (\{\, x \in \omega : T^{j+k} (x) \in 3B \,
    \})|}{|T^j (\omega)|} |\omega| \nonumber \\ &\leq \frac{3
    \gamma^{-2} C_3}{\delta} |B| |\omega| = C_4 |B| |\omega|.
\end{align}

This gives us enough control over those $\omega$ that satisfy $|T^j
(\omega)| \geq \delta_0$. We will also need some control of those
$\omega$ that do not satisfy this requirement, and we introduce the
following set of exceptional cylinders
\begin{multline*}
  F_i = \{\, \omega \in \mathscr{P}_{j_{i+1}} (a_1)| \cup \Omega_i : \not
  \exists \tilde{\omega} \in \mathscr{P}_{j_i + k} (a_1) | \cup \Omega_i,
  \ \tau \leq k \leq j_{i+1} - j_i, \\ \text{ s.t. } \tilde{\omega}
  \supset \omega \text{ and } |T^{j_i + k} (\tilde{\omega})|
  \geq \delta_0 \,\}.
\end{multline*}

The following lemma from Schnellmann's paper gives us enough control
of the sets $F_i$. The proof is in \cite{schnellmann}.

\begin{lemma} \label{lem:exceptional}
  There are numbers $n_{h,B}$, growing at most exponentially in $h$,
  such that
  \[
  |F_i| \leq \frac{(C_4 |B|)^h | \cup \Omega_0|}{h},
  \]
  for all $0 \leq i \leq h -1$ and $n \geq n_{h,B}$.
\end{lemma}

If $\omega \in \Omega_i \setminus F_i$, then for some $k$ we have
$|T^{j_i + k} (\omega)| \geq \delta_0$ and we get by \eqref{eq:hit3B2}
that
\[
| \{x \in \omega : T^{j_{i + 1}} (x) \in 3B \}| \leq C_4 |B|
|\omega|.
\]
This implies with Lemma~\ref{lem:exceptional} that
\[
|\cup \Omega_{i+1}| \leq C_4 |B| |\cup (\Omega_i \setminus F_i) | +
|F_i| \leq C_4 |B| |\omega_i| + \frac{(C_4 |B|)^h |\cup \Omega_0|}{h},
\]
provided $n \geq n_{h,B}$. Hence
\[
|\cup \Omega| \leq |\cup \Omega_h| \leq (C_4 |B|)^h |\cup \Omega_0| +
  |h \frac{(C_4 B|)^h |\cup \Omega_0|}{h} \leq 2 (C_4 |B|)^h | \cup
  \Omega_0|.
\]

Since
\[
|J_X| = |X(J)| + \frac{3C_2}{n} \leq (6C_2 + \sup_{a \in I} |X'(a)|)
|J|
\]
and $\Omega_0 = \{J_X\}$, we have
\[
|\cup \Omega_0| \leq (6C_2 + \sup_{a \in I} |X'(a)|) |J|.
\]
By Lemma~\ref{lem:q-to-one} we have $|\cup \Omega_J| \leq q C_2 |\cup
\Omega|$. Therefore
\[
|\cup \Omega_J| \leq (C|B|)^h |J|
\]
which implies \eqref{eq:Omegaestimate} and finishes the proof of
Theorem~\ref{the:typicalpoints}.

\end{document}